\documentclass[12pt]{article}

\usepackage{amssymb}
\usepackage{amsmath}
\usepackage{graphicx}
\usepackage{enumerate}
\usepackage{mathrsfs}
\usepackage{latexsym}
\usepackage{psfrag}
\usepackage{mathrsfs}
\usepackage{hyperref}

\newcommand{\qed}{\hfill $\square$\\
\vspace{0.1cm}}

\DeclareMathOperator* \ave {ave}
\newcommand{\hb}{\bar{h}}

\newcommand{\B}{\mathcal{B}}
\newcommand{\C}{\mathcal{C}}
\newcommand{\D}{\mathcal{D}}

\newcommand{\F}{\mathcal{F}}
\newcommand{\G}{\mathcal{G}}
\newcommand{\Ha}{\mathcal{H}}
\newcommand{\J}{\mathcal{J}}

\newcommand{\La}{{\rm La}}

\newcommand{\Pa}{\mathcal{P}}

\newcommand{\sC}{\mathscr{C}}

\newcommand{\V}{\mathcal{V}}

\newcommand{\op}{\mathbf{1}\oplus}
\newcommand{\po}{\oplus\mathbf{1}}

\newcommand{\Remark}{\noindent{\bfseries Remark.} }

\newtheorem{theorem}{Theorem}[section]
\newtheorem{lemma}[theorem]{Lemma}
\newtheorem{proposition}[theorem]{Proposition}
\newtheorem{corollary}[theorem]{Corollary}
\newtheorem{conjecture}{Conjecture}[section]
\newtheorem{question}[theorem]{Question}
\newenvironment{proof}{\noindent{\em Proof.}}{\qed}

\newcommand{\nchn}{\binom{n}{\lfloor \frac{n}{2}\rfloor}}

\newcommand{\lanp}{\La(n,P)}

\newcommand{\comments}[1]{}
\begin{document}
\title{Poset-free Families and Lubell-boundedness}

\author{Jerrold R. Griggs
\thanks{Department of Mathematics, University of South Carolina,
Columbia, SC, USA 29208 (griggs@math.sc.edu).
Research supported in part by a grant from the Simons Foundation (\#282896 to Jerrold Griggs),
by travel funding from National Taiwan University-NCTS/TPE, and by a long-term
visiting position at the IMA, University of Minnesota. } \and
Wei-Tian Li
\thanks{Department of Applied Mathematics, National Chung Hsing University,
Taichung, Taiwan 40227 (weitianli@dragon.nchu.edu.tw).  Supported in part by a
USC Graduate School Dissertation Fellowship, by NSC (102-2115-M-005-001),
and by MOST (103-2115-M-005 -003 -MY2).}}
\date{March 3, 2015\\}

\maketitle

\begin{abstract}

Given a finite poset $P$, we consider the largest size $\lanp$ of
a family $\F$ of subsets of $[n]:=\{1,\ldots,n\}$ that contains no
subposet $P$. This continues the study of the asymptotic growth of
$\lanp$; it has been conjectured that for all $P$, $\pi(P):=
\lim_{n\rightarrow\infty} \lanp/\nchn$ exists and equals a certain
integer, $e(P)$. This is known to be true for paths, and
for several more general families of posets, while for the simple diamond
poset $\D_2$, even the existence of $\pi$ frustratingly remains open.
Here we develop theory to show that $\pi(P)$ exists and equals the
conjectured value $e(P)$ for many new posets $P$.  We introduce a
hierarchy of properties for posets, each of which implies $\pi=e$,
and some implying more precise information about $\lanp$. The
properties relate to the Lubell function of a family $\F$ of
subsets, which is the average number of times a random full chain
meets $\F$.  We present an array of examples and constructions
that possess the properties.

\end{abstract}


\section{Introduction}\label{sec:Intro}

Let the Boolean lattice $\B_n$ denote the partially ordered set
(poset, for short)  $(2^{[n]},\subseteq)$ of all subsets of the
$n$-set $\{1,\ldots,n\}$. For finite posets $P=(P,\leq)$ and
$P'=(P',\leq')$, we say $P$ {\em contains\/} $P'$, if there exists an injection
$f\colon P' \to P$ that preserves the partial ordering.
We also say $P'$ is a {\em (weak) subposet\/} of $P$.
This means that whenever $u\leq' v$ in $P'$, we have $f(u)\leq f(v)$
in $P$~\cite[Chapter 3]{Sta}.

Collections $\F\subseteq 2^{[n]}$ that contain no subposet $P$
are said to be {\em $P$-free}. We are interested
in determining the largest size of a $P$-free family of subsets of
$[n]$, denoted $\lanp$. A foundational result of this sort from 1928,
Sperner's Theorem~\cite{Spe}, solves this problem for antichains,
which are families that contain no two-element chain.
More generally, let the path poset $\Pa_k$ consist
of $k$ totally ordered elements $a_1<\cdots<a_k$, which is
simply a chain of size $k$.
Erd\H{o}s~\cite{Erd} extended Sperner's Theorem
by giving the largest size of a family that contains no
$\Pa_k$.   It is simply the sum of the $k$ middle binomial
coefficients in $n$.
In recent years
Katona~\cite{DebKat,DebKatSwa,GriKat,KatTar} brought the attention of
researchers to the generalization of this problem, which is to
investigate $\lanp$ for many posets $P$.
For only a few posets does $\lanp$ behave as nicely as it
does for chains.
For most posets $P$, it appears to be far more challenging
to determine $\lanp$, and in fact it remains open, even
asymptotically.

Around 2007, when Griggs and Lu~\cite{GriLu} reviewed the
cases $P$ that had been studied, they observed that
while $\lanp$ may not be as simple as the sum of middle binomial
coefficients, it is at least true that $\lanp$ is asymptotic to an
integer multiple of $\nchn$.  This observation is also
implicit in the earlier work of Katona {\it et al}.
For a comprehensive survey of early work on the problem for
various posets $P$, please see~\cite{GriLiLu}.

Next, in 2008, Saks and Winkler pointed out
to Griggs a pattern of the known values of $\pi(P)$.
Griggs and Lu~\cite{GriLiLu} subsequently
formulated it by introducing some notation,
especially the parameter $e(P)$, as follows:
For a set $S$, the collection of all $k$-subsets of $S$ is
conventionally denoted by $\binom{S}{k}$.
In~\cite{GriLiLu},  $\B(n,k)$ is a family of subsets of $[n]$
of the $k$ middle sizes, $\binom{[n]}{\lfloor
(n-k+1)/2 \rfloor}\cup\cdots \cup\binom{[n]}{\lfloor (n+k-1)/2
\rfloor}$ or $\binom{[n]}{\lceil (n-k+1)/2 \rceil}\cup\cdots\cup
\binom{[n]}{\lceil (n+k-1)/2 \rceil}$.  So $\B(n,k)$ is one or two
possible families, depending on the parity of $n+k$.
Also, $\Sigma(n,k)$  is notation for $|\B(n,k)|$.
For a poset $P$, $e(P)$ denotes the maximum $k$ such that
for any integers $n$ and $s$, the family $\F=\binom{[n]}{s}\cup
\cdots\cup\binom{[n]}{s+k-1}$ is $P$-free. In particular, the
union $\B(n,k)$ of $k$ middle levels in $\B_n$ does not contain
$P$ as a subposet.

For instance, the {\em butterfly poset} $\B$ of elements $A_1$,
$A_2$, $B_1$, and $B_2$ with $A_i<B_j$ for $i,j=1,2$ is not
contained in the union of two consecutive levels in the Boolean
lattice, while of course the union of three middle levels does
contain $\B$ for $n\ge 3$. One gets that $e(\B)=2$.
Since the family $\B(n,e)$, where $e=e(P)$, contains no $P$, it is
clear that when $\pi(P)$ exists, it must be at least $e(P)$.

The Griggs-Lu Conjecture is that $e(P)$ is the limiting value.

\begin{conjecture}\label{conj:GriLu}{\rm\cite{GriLiLu}}
For any poset $P$, the limit
$\pi(P):=\lim_{n\rightarrow\infty}\frac{\lanp}{\nchn}$ exists.
Moreover, its value is the integer $e(P)$.
\end{conjecture}

Griggs and Lu presented several new families of posets for which
the their conjecture holds~\cite{GriLu}, and they improved the known
bounds on $\pi(P)$ for some families for which the existence of
$\pi(P)$ is still not settled.  One noteworthy discovery is due
to Bukh~\cite{Buk}, who proves the existence of $\pi(P)$ for all
tree posets.
Moreover, for tree posets $\pi(P)$ is the height of $P$
minus one, which is indeed $e(P)$.

Still it remains a daunting problem to
obtain $\pi(P)$, even for certain small posets $P$.
The most-studied case is the {\em diamond poset\/} $\D_2$,
consisting of four elements $A,B,C,D$ with $A<B,C$ and $B,C<D$.
While the conjectured value of $\pi(\D_2)$ is $e(\D_2)=2$, a
series of studies has so far only brought the upper bound down to
$2.25$~\cite{KraMarYou}.
The existence of $\pi(\D_2)$ remains unproven.
It appears that additional tools must be developed.
While many researchers continue to look for improved upper bounds for
the diamond poset, this paper takes a different approach, which is
to develop theoretical tools to greatly expand the list of
posets that satisfy the Griggs-Lu Conjecture.

In their subsequent work with Lu~\cite{GriLiLu}, the authors
learned that for a $P$-free family $\F$, it can be valuable to
consider the average number of times a random full chain of
subsets of $[n]$ intersects $\F$, which they called the Lubell
function of $\F$, denoted $\hb(\F)$.
The present authors~\cite{GriLi} described a
``partition method" for using the Lubell function to derive simple
new proofs of several fundamental poset examples satisfying the
conjecture.

In this paper, we extensively expand the approach of bounding the
Lubell function. We introduce a new series of properties of posets
$P$ for which the conjecture above is satisfied. For posets $P$
satisfying the most restrictive of these properties, named here
 uniform L-boundedness, it was already shown in~\cite{GriLiLu,Li} that
not only does $\lanp$ satisfy the asymptotic conjecture, but it is
exactly determined for general $n$.  Moreover, for such $P$
one can describe all extremal families, just as in the early Erd\H os
result for path posets $\Pa_k$.

The present treatment introduces a series of 
properties, called $m$-L-boundedness, for integers $m\ge0$;
When $m=0$ it is uniform L-boundedness.
As $m$ increases, the properties get successively weaker
(more easily satisfied).
What happens is that families of subsets that contain some of the
comparatively few sets at the top and bottom of the Boolean lattice
may have large Lubell function value, even though such sets
contribute little to the size of the family.
For studying $\lanp$ asymptotically, it makes sense to focus on
$P$-free families of subsets away from the top and bottom
of the Boolean lattice.
We do this with the $m$-L-bounded properties, obtaining many
more posets that satisfy the $\pi(P)$ conjecture.
Uniform L-boundedness is extended in another way, with
properties called lower L-bounded and upper L-bounded,
that also imply the conjecture.

Section~\ref{sec:LUB} reviews the required poset terminology
and concepts related to the Lubell function.  The new properties
and their connections to the conjectures are developed in
Section~\ref{sec:LBP}.

Section~\ref{sec:LAR} introduces a notion connected with
$e(P)$, called a large interval of a poset.  This plays an
important role in the constructions discussed later.

A new class of posets is introduced in Section~\ref{sec:FAN},
called fan posets, which are simply wedges of paths.  These
include the fork posets $\V_r$ previously investigated by Katona {\em
et al.}, for integers $r\ge1$, with elements $A<B_i$ for $1\le i\le
r$. We determine which fans are L-bounded, and give examples for
all $m\ge1$ of a fan that is $m$-L-bounded but not $(m-1)$-L-bounded.
Since all fan posets are trees, and since trees satisfy the
$\pi(P)=e(P)$ conjecture by Bukh's Theorem~\cite{Buk}, it
follows that all fans satisfy the conjecture. We give a simpler
direct proof of this for fans using the Lubell properties.

Constructive methods to generate a surprising variety
of $m$-L-bounded and lower-L-bounded posets from old ones
are described in
Sections~\ref{sec:CONS} and~\ref{sec:BLUP}. All posets generated
in this way satisfy the $\pi=e$ conjecture.  Some of the many
interesting problems for further research are discussed in the
last section.

Some of the problems are based on the numerous thoughtful
questions and suggestions of the anonymous referees, who
also contributed
greatly to improving the paper's presentation.

\section{The Lubell Function and Three Poset Parameters}\label{sec:LUB}

Recall some standard poset notions.  For elements $a\le b$
in poset $P$, a (closed) {\em interval} $[a,b]\subset P$ is the
subposet of $P$ consisting of elements $c$ such that $a\le c\le
b$. Note that if $P$ is the Boolean lattice $\B_n$, then an
interval $[A,B]$ has the same structure as $\B_{|B|-|A|}$.

The {\em dual\/} of a poset $P=(P,\le)$ is the poset
$d(P)=(P,\le_d)$ such that $x\le_d y$ in $d(P)$ if any only if
$y\le x$ in $P$.

An element $x$ of a poset $P$ is $\hat{0}$ (resp., $\hat{1}$), if
for every element $p\in P$, $x\le p$ ($x\ge p$, resp.).

We introduce notation for the {\em filter} or {\em up-set} (resp.,
{\em ideal} or {\em down-set}) generated by an element $p\in P$:
Let $\{p\}^+$ (resp., $\{p\}^-$) denote the sets $\{q\in P\mid
q\ge p\}$ and $\{q\in P\mid q\le p\}$, resp.

The {\em ordinal sum\/} $P_1\oplus P_2$
of disjoint posets $P_1, P_2$ is the
set $P_1\cup P_2$, ordered by $x\le y$
if $x\in P_1$ and $y\in P_2$, or if $x,y$ are in the same $P_i$
with $x\le y$.   We denote by $\mathbf{1}$ the single element poset,
and $k\mathbf{1}$ denotes the $k$-element antichain.

Fix a family $\F\subseteq 2^{[n]}$.  Let $\sC:=\sC_n$ denote the
collection of all $n!$ full (maximal) chains $\emptyset \subset
\{i_1\} \subset \{i_1,i_2\} \subset \cdots \subset [n]$ in the
Boolean lattice $\B_n$. The average
number of times a random chain $\C\in \sC$ meets $\F$ gives an upper
bound on $|\F|$.
The {\em height\/} of $\F$ is
\[h(\F):= \max_{\C\in\sC} |\F\cap \C|.\]
Following~\cite{GriLiLu} we define the {\em Lubell function\/}
of $\F$ by
\[\hb(\F)=\hb_n(\F):=\ave_{\C\in\sC} |\F\cap \C|.\]
The Lubell function bounds the size of a family:

\begin{lemma}\label{lub}{\rm\cite{GriLiLu}}
Let $\F$ be a collection of subsets of $[n]$. Then $\hb(\F)
=\sum_{F\in \F} 1/\binom{n}{|F|}\ge |\F|/\nchn$.
\end{lemma}

This lemma is an extension of the heart of Lubell's elegant proof
of Sperner's Theorem~\cite{Lub}. We see that $\hb(\F)$ can be
viewed as a weighted sum, where each set $F$ has weight
$1/\binom{n}{|F|}$. To maximize $|\F|$ over families $\F$ of given
weight, the sets in the family must have weights as small as
possible.

By computing $\hb(\F)$ for all $P$-free $\F$, we obtain an upper
bound on $\lanp$, so also on $\pi(P)$, if it exists. Let
$\lambda_n(P)$ be the maximum value of $\hb(\F)$ over all $P$-free
families $\F\subset 2^{[n]}$. Then, $\lanp/\nchn\le \lambda_n(P)$.
We define $\lambda(P):=\lim_{n\rightarrow\infty} \lambda_n(P)$, if
this limit exists.  Collecting what we have observed, we get that
\[
e(P)\le\pi(P)\le \lambda(P),
\]
if both limits $\pi, \lambda$ exist.

In the following sections we will see many
posets for which $\pi(P)=\lambda(P)$.  On the other hand, there
are posets with $\pi(P)<\lambda(P)$:  The smallest example
is the fork $\V_2$, the poset on three elements $A<B$ and $A<C$.
It is known that $\pi(\V_2)=1$, while easily
$\lambda(\V_2)=2$, since the family $\{\emptyset, [n]\}$ shows
that $\lambda_n(\V_2)=2$ for all $n$.

We continue to believe that $\pi(P)=e(P)$ for general $P$.
While $e(P)$ exists for any poset $P$, we do not know how to
determine it in general. The height $h(P)$ alone
is not sufficient to determine $e(P)$:  Consider the
{\em $k$-diamond} poset $\D_k$, $k\ge2$, which has elements
$\{A,B_1,\ldots,B_k,C\}$ ordered by $A<B_i<C$ for $1\le i\le k$
It can be viewed as the ordinal sum of antichains,
$\mathbf{1}\oplus k\mathbf{1}\oplus\mathbf{1}$. Jiang and Lu~\cite{GriLiLu}
independently observed that, although the diamonds $\D_k$ have height
three, $e(\D_k)$ becomes arbitrarily large as $k$ grows.

\section{Lubell-bounded Posets}\label{sec:LBP}

We next introduce and investigate properties based on the Lubell
function that are useful for obtaining posets $P$ that satisfy the
$\pi(P)=e(P)$ conjecture. With Lu~\cite{GriLiLu}, we considered
posets $P$ for which
\[
\hb_n(\F)\le e(P)
\]
for every $n$ and $P$-free family $\F$ of subsets of $[n]$.
We say such posets are {\em uniformly L-bounded}, with L for Lubell,
since the Lubell functions of $P$-free families are bounded by
$e(P)$, for every $n$.  Since $\B(n,e)$ is
$P$-free for $e=e(P)$ and $n\ge e-1$, it follows that uniformly
L-bounded posets $P$ satisfy $\lambda_n(P)=e(P)$ for all $n\ge
e-1$, and so for such $P$, $\lambda(P)=e(P)$.

For the diamond poset $\D_2$ the existence of the limit
$\pi(\D_2)$ remains elusive.  The diamond $\D_2$ is certainly not
uniformly L-bounded, since there are $\D_2$-free families $\F$ for
which $\hb_n(\F)>2.25$ for large $n$, as compared to $e(\D_2)=2$.
However, for ``most" values $k>2$ the diamond $\D_k$ is uniformly
L-bounded.

Also introduced in~\cite{GriLiLu} is the
{\em harp} poset $\Ha(\ell_1,\ldots,\ell_k)$, consisting of $k$
chains $\Pa_{\ell_1},\ldots,\Pa_{\ell_k}$ with their minimum
elements identified and their maximum elements identified.
Provided the chain sizes are distinct (so satisfy
$\ell_1>\cdots>\ell_k\ge3$), the harp is uniformly
L-bounded.

Here is what we could show about the largest $P$-free
families for uniformly L-bounded $P$.  Note that beyond
determining $\pi(P)$, we know $\lanp$ exactly, and we even know the
extremal $P$-free families.

\begin{theorem}\label{thm:ULB}{\rm\cite{GriLiLu}}
Let $P$ be a poset that is uniformly L-bounded.  Let $e=e(P)$.
Then for all $n$, $\lanp=\Sigma(n,e)$, and so $\pi(P)=e$.
Moreover, if $\F$ is a $P$-free family of subsets of $[n]$ of
maximum size, $\F$ must be $\B(n,e)$.
\end{theorem}

A problem with the Lubell function is the large contribution to it
(one each) from the empty set or the full set $[n]$.  For example,
consider the butterfly poset $\B$:  For general $n$ the family
$\F$ consisting of  $\emptyset, [n],$ and the singletons is a
(small) $\B$-free family with $\hb(\F)=3>e(P)=2$, so the butterfly
is not uniformly L-bounded.

For the diamond poset $\D_2$, examples show that $\hb_n(\F)$ is
at least $2.25$ as $n\rightarrow\infty$, which prevents this
approach from proving that the $\pi(\D_2)$ is 2.  However, the
examples with large Lubell function value involve very small sets
(or their complements).  It may be that for families without
those very small or very large sets, the Lubell function does
tend to 2.  That would be good enough for us, since comparatively
very few sets are small or large, not enough sets to affect
$\pi(D_2)$.
We then introduce properties that avoid these sets.

To investigate asymptotically the maximum size $\lanp$ of
$P$-free families, it makes sense to restrict attention first to
families that do not contain $\emptyset, [n]$.  If the Lubell
function of such families stays small, then we still get good
bounds on $\lanp$, at least asymptotically.  Let us say $P$ is
{\em centrally L-bounded}, if for all $n$, $\hb_n(\F)\le e(P)$, for
all $P$-free families $\F$ of proper subsets of $[n]$, that is,
$P$-free families that exclude $\emptyset, [n]$.
The butterfly $\B$ is an example of a centrally L-bounded
poset, as we showed in~\cite{GriLi}.
Here is the analogue of Theorem~\ref{thm:ULB}
for centrally L-bounded posets.

\begin{theorem}\label{thm:CLB}
Let $P$ be a centrally L-bounded poset, and let $e=e(P)$. For
$n\ge e+3$,
\[\lanp=\Sigma(n,e).
\]
Hence, $\pi(P)=e$.
\end{theorem}

\begin{proof}
Given a centrally L-bounded poset $P$, suppose that $\F$ is a
$P$-free family of subsets of $[n]$ of maximum size, where $n\ge
e+3$. If $\F$ contains neither $\emptyset$ nor $[n]$, then by
definition, $\hb(\F)\le e$, and easily by Lemma~\ref{lub} we get
$|\F|\le\Sigma(n,e).$

Next consider if $\F$ contains $\emptyset$, but not $[n]$. We know
that $\hb(\F)\le e+1$, since $P$ is centrally L-bounded,
and since $\F$ with  $\emptyset$ removed remains $P$-free.
Then $\F$ cannot
contain all singleton subsets of $[n]$; otherwise, we could form
$\F'=\F\setminus(\{\emptyset\}\cup\binom{[n]}{1})$, and have
$\hb(\F')\le e-1$. But then, $|\F|=1+n+|\F'|\le
1+n+\Sigma(n,e-1)<\Sigma(n,e)$.  This contradicts that
$|\F|=\lanp\ge\Sigma(n,e)$. So some singleton subset $\{i\}$ is
missing from $\F$, and we may obtain a new family $\F''$ from $\F$
by replacing $\emptyset$ by $\{i\}$.

We claim $\F''$ is $P$-free.  Otherwise, $\F''$ contains $P$,
and  $\{i\}$ must be a minimal element of $P$. All other
elements in $P$ are in $\F\setminus\{\emptyset\}$. Then replacing
$\{i\}$ by $\emptyset$, we see that $\F$ itself contains $P$, a
contradiction. By central L-boundedness of $P$, $\hb(\F'')\le e$.
We deduce the desired bound from $|\F|=|\F''|\le \Sigma(n,e)$.

Note that when $[n]\in\mathcal{F}$ and $\emptyset\not\in\mathcal{F}$
the proof is similar.
Finally consider $\F$ containing both $\emptyset$ and $[n]$.
Suppose $n\neq 4$ or $e\neq 1$.  Then if
$\binom{[n]}{1}\subset\F$, we obtain $|\F|\le
2+n+\Sigma(n,e-1)<\Sigma(n,e)$, which contradicts $\F$ having
maximum size. Similarly, $\F$ cannot contain all $(n-1)$-subsets
of $[n]$. We can do replacements as before to obtain a new
$P$-free family that contains neither $\emptyset$ nor $[n]$ and
has size as large as $\F$. Applying the central L-boundedness of
$P$ to this new family, $|\F|=\Sigma(n,e)$.

The last case is $n=4$, $e=1$, and $\F$ contains both $\emptyset$
and $[4]$. If $\F$ contains all singleton subsets of $[4]$, then
$\F=\{\emptyset,[4]\}\cup\binom{[4]}{1}$. Let $\F'=(\F\setminus
\{[4]\})\cup\{ S\}$ for some 3-subset $S$. On the one hand, $\F'$ is
$P$-free since if $S$ represents some element of $P$, then $[4]$
could be the same element of $P$. With other elements of $P$ in
$\F\setminus \{[4]\}$,  we conclude that $\F$ contains $P$, which
contradicts our assumption. On the other hand, $\F'$ must contain
$P$ since $\hb(\F'\setminus\{\emptyset\})=1\frac{1}{4}>e(P)$. The
dilemma is caused by the assumption $\binom{[4]}{1}\subset\F$.
Hence $\F$ cannot contain all singleton subsets. By the same
reasoning, $\F$ cannot contain all 3-subsets of $[4]$. Therefore
we can just replace $\emptyset$ and $[4]$ by a singleton
subset and a 3-subset, resp. Again, we  conclude that
$|\F|=\Sigma(n,e)$.
\end{proof}

This property of a poset being centrally L-bounded generalizes how
we showed in~\cite{GriLi} that when $P$ is the butterfly,
$\lanp=\Sigma(n,e)$ for all sufficiently large~$n$.   Note that,
unlike the more restrictive class of uniformly L-bounded posets,
it need not hold that $\lanp=\Sigma(n,e)$ for {\em all\/} $n$ - it can
fail for small $n$. Indeed, we see this for the butterfly, where
$e=2$: For $n=2$, $\B_2$ is a $\B$-free family of size 4,
which is more than $\Sigma(2,2)=3$.

A further distinction is that extremal families for centrally
L-bounded $P$ are not restricted to the middle level families
$\B(n,e)$, as they are when $P$ is uniformly L-bounded. For
instance, when $P$ is the butterfly, a construction of
DeBonis {\it et al.\/}~\cite{DebKatSwa} for $\B_4$ is
\[
\F=\{\{1\},\{2\},\{1,3,4\},\{2,3,4\}\}\cup\binom{[4]}{2},
\]
 which is a different butterfly-free family of maximum size ($e=2$ and
$|\F|=\Sigma(4,2)$).
For $n\ge 5$, they prove a largest butterfly-free family must be
$\B(n,2)$.

Another instructive example is the J poset $\J$  studied in~\cite{Li}.
It has four elements $A_1<A_2<A_3$ and $A_1<B$.  It is one of the fan
posets introduced in the next section, where we will see that
it is centrally L-bounded.  Again, we have $e=2$, and in $\B_2$ the
family $\F=\{\emptyset,\{1\},\{1,2\}\}$ is a largest $\J$-free
family, of size $\Sigma(2,2)$, that is not $\B(2,2)$.

For general centrally L-bounded posets $P$,
are such exceptions, where there are largest $P$-free families besides
taking the middle levels, only possible for small $n$?  We address this
in the closing section of the paper.

Next we explore further weakening the uniformly L-bounded
property,  by
further limiting the families for which the Lubell
function is required to be bounded.
The idea is that while sets with very few elements--or
dually very many elements--can cause the Lubell
function to be large, the number of such sets is
negligible compared to the size of a largest $P$-free family,
viewed asymptotically as $n$ grows.
It means that to investigate the asymptotics of
$\lanp/\nchn$ as $n$ grows, it suffices to consider
$P$-free families that contain no small or large sets.
The advantage is that by eliminating the small and large
sets from a family, we may be able to lower our upper
 bound on the Lubell function, and, hence, on the size of
 the family.

We could actually exclude all sets that have size
differing by more than $\sim\log{n}\sqrt{n}$ from the
middle size, $n/2$, but no one has yet found a good way
to take proper advantage of such a size restriction for
the most challenging posets $P$, such as the diamond.
However, there are posets $P$ for which even a very mild
restriction on size is sufficient to deduce that
$\pi(P)=e(P)$.

We introduce a family of poset properties, indexed by
integers $m\ge0$, each of which implies $\pi(P)=e(P)$.
For given $m$, we ignore the
subsets in both the bottom $m$ and top $m$ levels of the Boolean
lattice:.  We say $P$ is {\em $m$-L-bounded}, if for all $n$,
$\hb_n(\F)\le e(P)$ for all $P$-free families $\F$ of subsets of
sizes in $[m,n-m]$. Then a poset is 0-L-bounded means it is
uniformly L-bounded, while 1-L-bounded means it is centrally
L-bounded.  As $m$ increases, the increasing restriction on the
families $\F$ that must satisfy the Lubell function condition,
means that more posets potentially have the property.  We develop the
theory for these properties, and verify the asymptotic conjecture
$\pi=e$ for posets that possess them.  Interestingly, in
the section on fan posets, we shall give examples for all $m\ge1$ of
posets that are $m$-L-bounded but not ($m-1$)-L-bounded.

\begin{proposition}\label{thm:mLB}
Let integer $m\ge0$.  Let $P$ be an $m$-L-bounded poset, and let
$e=e(P)$.  Then for all $n$,
\[
\lanp\le \Sigma(n,e)+2\sum_{i=0}^{m-1}\binom{n}{i}.
\]
Hence, $\pi(P)=e$.
\end{proposition}

\begin{proof}
For an $m$-L-bounded poset $P$, the bound on $\lanp$ follows by
removing the tails (bottom and top $m$ levels) of a $P$-free
family $\F$ that achieves $\lanp$, since the family that remains
is uniformly L-bounded.  Asymptotically, $\Sigma(n,e)\sim e\nchn$,
while the number of sets at the tails (the $m$ smallest and
largest sizes) is only $O(n^{m-1})$.
\end{proof}

To capture $m$-L-boundedness for general $m$, we say
$P$ is {\em L-bounded,} if it is $m$-L-bounded for some $m$.
This class is then the union of the classes of $m$-L-bounded
posets over all $m$, and the Proposition above applies.

\begin{corollary}\label{thm:LB}
If poset $P$ is L-bounded, then $\pi(P)=e(P)$.
\end{corollary}

It is not surprising that not every poset is L-bounded.
Indeed, consider the three-element poset $\V_2$.
In Section~\ref{sec:LUB}, we proved that
$\lambda_n(\V_2)=\max\hb_n(\F)=2$
over all $\V_2$-free families $\F$ of subsets of $[n]$.
Since this is larger than $e(\V_2)=1$,
$\V_2$ is not uniformly L-bounded.
Is it $m$-L-bounded for some positive integer $m$?
The answer is still no, by the following construction:
Fix any $m\ge 2$, and consider the following family:
\[
\F_n:=\{F: |F|=m, \forall x,y\in F, x\equiv y\pmod 2)\}
\cup \{F: |F|=m+1, \exists x,y\in F, x\not\equiv y\pmod 2)\}.
\]
One can check that its conjugate (set of complements),
$\overline{\F}_n=\{[n]-F\mid F\in\F_n\}$, is $\V_2$-free, and we have
\[
\hb_n(\overline{\F}_n)=\frac{\binom{\lceil n/2\rceil}{m}+\binom{\lfloor n/2\rfloor}{m}}{\binom{n}{m}}+
\left(1-\frac{\binom{\lceil n/2\rceil}{m+1}+\binom{\lfloor n/2\rfloor}{m+1}}{\binom{n}{m+1}}\right)
\sim 1+\left(\frac{1}{2}\right)^m>1=e(\V_2).
\]
This means that for fixed $m$, the maximum Lubell function
value is bounded strictly above $e(P)$ as $n\rightarrow\infty$.

Nevertheless, we have a method to show that $\pi(\V_2)=e(\V_2)$,
by discarding more elements from the Boolean lattice,
as we now describe.
We introduce a
different weakening of uniform L-boundedness.
We say a poset $P$ is {\em lower-L-bounded\/} if, for any
numbers $\beta\in (\frac{1}{2},1)$ and $\varepsilon>0$, there
exists $N:=N(\beta,\varepsilon)$ such that for all $n\ge N$,
\[
\hb(\F)\le e(P)+\varepsilon,
\]
for all $P$-free families of subsets of $[n]$ of sizes less than
$\beta n$. We are interested as well in the dual property:
Say poset $P$ is {\em upper-L-bounded\/} if, for any numbers
$\alpha\in (0,\frac{1}{2})$ and $\varepsilon>0$, there exists
$N:=N(\alpha,\varepsilon)$ such that for all $n\ge N$,
\[
\hb(\F)\le e(P)+\varepsilon,
\]
for all $P$-free families of subsets of $[n]$ of sizes greater than
$\alpha n$.

\begin{proposition}\label{thm:LLB}
Let $P$ be a lower-L-bounded or upper-L-bounded poset.  Then
$\pi(P)=e(P)$.
\end{proposition}

\begin{proof}
Since the properties are dual, it suffices to prove this for any
lower-L-bounded poset $P$.  For each $n$, let $\F_n$ be a largest
$P$-free family of subsets of $[n]$. Fix some
$\beta\in(\frac{1}{2},1)$. Partition $\F_n$ into $\F'_n$ and
$\F''_n$ such that $\F'_n$ has sets of $\F_n$ of sizes at most
$\beta n$ and $\F''_n=\F_n\setminus\F'_n$.
Shannon's Theorem~\cite[page 256]{AloSpe} gives
$\sum_{i=0}^{\alpha n} \binom{n}{i}=O(\frac{1}{n^2})\nchn$, for any
constant $0<\alpha<\frac{1}{2}$.
For any $\varepsilon>0$, for $n\ge N$, we have
\begin{align*}
\frac{|\F_n|}{\nchn}
&\le \sum_{F\in\F'_n}\frac{1}{\binom{n}{|F|}}+\frac{|\F''_n|}{\nchn}\\
&\le \hb_n(\F'_n)+O\left(\frac{1}{n^2}\right)\\
&\le e(P)+\varepsilon +O\left(\frac{1}{n^2}\right).
\end{align*}
Since $\varepsilon$ is arbitrary, this implies $\pi(P)=e(P)$.
\end{proof}

We shall see a class of posets, including $\V_2$, that are
lower-L-bounded, but not L-bounded, in Section~\ref{sec:FAN}.
Constructions of lower-L-bounded posets are presented in
Section~\ref{sec:BLUP}.


\section{Large Intervals}\label{sec:LAR}

We next introduce a poset structure that plays a role
in this theory of L-boundedness by helping us to
estimate the three poset quantities $\pi(P), e(P), \lambda(P)$.
We say an interval $I=[a,b]$ in $P$ is a {\em
large interval\/}, if it is a maximal interval with $e(I)=e(P)$.
Here are several observations, all relating to decomposing a
poset $P$ into two smaller pieces.

\begin{lemma}\label{lem:parameters}
Suppose  $P$ is a poset with element $p$, such that $P=P_1\cup
P_2$, where $P_1\cap P_2=\{p\}$ and $P_2=\{p\}^+$. Then
\begin{description}
\item{(i)} $\lambda_n(P)\le \lambda_n(P_1)+\lambda_n(P_2)$,
\item{(ii)} $\La(n,P)\le \La(n,P_1)+\La(n,P_2)$, and
\item{(iii)} $\pi(P)\le \pi(P_1)+\pi(P_2)$, if they exist.
\item{(iv)} Further, if $p$ is the maximal element of a
large interval $I$ of $P_1$, or if $P_1=\{p\}^-$
then $e(P)= e(P_1)+e(P_2)$.
\end{description}
\end{lemma}

\begin{proof}
Given $P, P_1, P_2$ as in the statement,
suppose $\F$ is a $P$-free family. Let $\F_1:=\{S\in\F\mid
\F\cap[S,[n]]\mbox{ contains }P_2\}$ and
$\F_2:=\F\setminus\F_1$.

We first show that $\F_1$ is $P_1$-free.
For, suppose $\F_1$ contains a subposet $P_1$, and let
$S\in\F_1$ be the set that represents the element $p$
of $P_1$.
Let $T$ be a maximal set in $\F_1$ containing $S$.
By maximality,
$(\F\setminus\{T\})\cap[T,[n]]$ is contained in $\F_2$.
Furthermore, by definition of $\F_1$,
$\F\cap[T,[n]]$ contains $P_2$ as a subposet.
Thus, $\F$ contains $P$, which contradicts our assumption.

Next we observe that $\F_2$ is $P_2$-free.
For if it did contain a subposet $P_2$, the set $S\in\F_2$
corresponding to element $p$ would have to be in $\F_1$,
a contradiction.

Since every family can be partitioned as above, we have
$\hb(\F)\le\hb(\F_1)+\hb(\F_2)$. This implies (i). Furthermore,
$|\F|=|\F_1|+|\F_2|$ implies $\lanp\le\La(n,P_1)+\La(n,P_2)$,
which is (ii). Then, (iii) is a consequence of (ii).

For (iv), let $e(P_1)=e_1$ and $e(P_2)=e_2$.  First we show
$e(P)\ge e_1+e_2$. Suppose for some $n,s$ the family
$\bigcup_{i=0}^{e_1+e_2-1}\binom{[n]}{s+i}$ contains $P$, where
$0\le s\le n-e_1-e_2+1$. Let $S$ be the set in the family
corresponding to element $p\in P$. Its size $|S|$ cannot be less
than $s+e_1$ nor greater than $s+e_1-1$:  Else, it would imply
$e(P_1)<e_1$ or $e(P_2)<e_2$. Hence, the family is $P$-free, and
$e(P)\ge e_1+e_2$.

It remains to show $e(P)\le e_1+e_2$. By the definition of $e_i$,
for $i=1,2$
we can find an integer $n_i$ large enough so that there exists
$P_i$ in the family of  $e_i+1$ consecutive levels
$\bigcup_{j=0}^{e_1}\binom{[n_i]}{s_i+j}$ of $\B_{n_i}$.
Let $n_0=n_1+n_2+e_1+e_2$ and $s_0=s_1+s_2+e_1+e_2$.
Consider the family $\bigcup_{i=0}^{e_1+e_2}\binom{[n_0]}{s_0+i}$.
An interval $[S_1,T_1]$ in $[n_0]$, with $|S_1|=s_1+e_1+e_2$ and
$|T_1|=s_1+e_1+e_2+n_2$, is the same poset as the
Boolean lattice $\B_{n_2}$. Hence, $[S_1,T_1]\cap
(\bigcup_{i=0}^{e_1+e_2}\binom{[n_0]}{s_0})$ contains $P_2$.
Furthermore, the element $p$ in $P_2$ could be chosen as a set of
size at least $s_1+s_2+e_1+e_2$. On the other hand, an interval
$[S_2,T_2]$ with $|S_2|=s_2+e_2$ and $|T_2|=s_2+e_2+n_1$ is the
same as the Boolean lattice $\B_{n_1}$. Thus, $[S_2,T_2]\cap
(\bigcup_{i=0}^{e_1+e_2}\binom{[n_0]}{s_0})$ contains $P_1$, and
all elements in $P_1$ are sets of size at most
$s_1+s_2+e_1+e_2$. By relabelling the elements, if needed,
one sees that the family
$\bigcup_{i=0}^{e_1+e_2}\binom{[n_0]}{s_0+i}$ contains $P$ as a
subposet. That is, for sufficiently large $n$, some
$e_1+e_2+1$ consecutive levels contain $P$, which implies the
desired inequality on $e(P)$.  This proves (iv).
\end{proof}

Note that the second part  of Lemma~\ref{lem:parameters} (iv),
the case that $P_1=\{p\}^-$, has
 been discovered independently by Burcsi and
Nagy~\cite{BurNag}.

Our purpose in studying large intervals is to find when the
parameter $e(P_1\cup P_2)$ has the ``additive property'' of
 Lemma~\ref{lem:parameters} (iv). Unfortunately, not all
posets contain a large interval.  For example, every maximal
interval of the butterfly poset $\B$ is a $\Pa_2$, but
$e(\B)>e(\Pa_2)=1$.  So poset $\B$, which is centrally L-bounded
(hence, L-bounded), has no large interval. Nevertheless, if an
L-bounded poset contains a large interval, then it is unique.

\begin{proposition}\label{lem:unilarge}
Let $P$ be an L-bounded poset. Then $P$ contains at most one large
interval. Furthermore, if an $m$-L-bounded $P$ contains a large
interval $I$, then $I$ is itself $m$-L-bounded.
\end{proposition}

\begin{proof}
Let $P$ be an L-bounded poset, say it is $m$-L-bounded. Assume
that $P$ contains two large intervals $I_1=[a_1,b_1]$ and
$I_2=[a_2,b_2]$ such that $e:=e(P)=e(I_1)=e(I_2)$. Without loss of
generality, we assume $b_1\neq b_2$ (or we may instead consider
the dual of $P$). Consider the Boolean lattice $\B_n$ with $n \ge
2m+e$.

Define the family
$\G:=(\bigcup_{i=1}^{e}\binom{[n]}{m+i-1})\cup\{S\}$, where $S$ is
any set of size $m+e$.  We claim $\G$ is $P$-free.
For if $\G$ contains $P$, it contains $I_1$.
Then $S$ must be $b_1$ since the family of
any $e$ consecutive levels does not contain $I_1$.
Similarly, $\G$ contains $I_2$, and $S$ must be $b_2$,
which contradicts  $b_1 \neq b_2 $.
 We conclude that $\G$ does not contain $P$.

The $P$-free family $\G$ satisfies
$\hb(\G)=e+1/\binom{n}{m+e}$ and every set in $\G$ has sizes
in $[m,n-m]$. This violates the assumption of the
$m$-L-boundedness for $P$. Therefore, $P$ cannot contain two large
intervals.

Now suppose that $P$ contains a large interval $I$. Let $\F$ be
any $I$-free family such that every set in $\F$ has size in
$[m,n-m]$. Since $\F$ is $I$-free, it is also $P$-free. By the
$m$-L-boundedness of $P$, we have $\hb(\F)\le e=e(I)$. So $I$
is $m$-L-bounded.
\end{proof}

\Remark The proof above also implies that if $P$ has two large
intervals, then $\lanp$ is strictly greater than $\Sigma(n,e)$ for
all large enough $n$:  One can take $\B(n,e)$ and one extra set
 to form a $P$-free family.

\section{Fans}\label{sec:FAN}

For our theory of Lubell boundedness, it turns out to be very
interesting to study the natural common generalization of the
fork posets $\V_r$
(studied by Katona {\em et al.}) and the poset $\J$.
For $\ell_1\ge\cdots\ge\ell_k\ge 2$, we
say the poset obtained by identifying the minimum
elements of $k$ chains $\Pa_{\ell_1},\ldots,\Pa_{\ell_k}$
is called the {\em fan} $\V(\ell_1,\ldots,\ell_k)$ .  That
is, a fan is simply a wedge of paths.  We then have
$\V_r=\V(2,\ldots,2)$, where there are $r$ 2's, while
the J poset $\J=\V(3,2)$.
Fans are similar to harps introduced in~\cite{GriLiLu}, except
in harps the maximum
elements of the chains are also identified. So a harp can be
viewed as a suspension of paths (see Figure~\ref{fig:FAN}).

\begin{figure}[h]
\begin{center}
\begin{picture}(200,60)
\put(50,0){\circle*{4}}
\put(35,15){\circle*{4}}
\put(20,30){\circle*{4}}
\put(5,45){\circle*{4}}
\put(50,15){\circle*{4}}
\put(50,30){\circle*{4}}
\put(65,15){\circle*{4}}
\put(80,30){\circle*{4}}
\put(50,0){\line(-1,1){45}}
\put(50,0){\line(0,1){30}}
\put(50,0){\line(1,1){30}}
\put(130,15){\circle*{4}}
\put(130,30){\circle*{4}}
\put(150,0){\circle*{4}}
\put(160,15){\circle*{4}}
\put(160,30){\circle*{4}}
\put(150,45){\circle*{4}}
\put(180,22.5){\circle*{4}}
\put(150,0){\line(-4,3){20}}
\put(150,45){\line(-4,-3){20}}
\put(130,15){\line(0,1){15}}
\put(150,0){\line(2,3){10}}
\put(150,45){\line(2,-3){10}}
\put(160,15){\line(0,1){15}}
\put(150,45){\line(5,-4){30}}
\put(150,0){\line(5,4){30}}

\end{picture}
\end{center}
\caption{The fan $\V(4,3,3)$ and the harp $\Ha(4,4,3)$.}\label{fig:FAN}
\end{figure}

In our notation, Bukh's Tree Theorem~\cite{Buk} tells us that
for posets $P$ for which the Hasse diagram is a tree,
$\pi(P)=e(P)=h(P)-1$,  where $h(P)$ is the height (cf.~\cite{GriLiLu}).
It follows that for the fan
$P=\V(\ell_1,\ldots,\ell_k)$, where $\ell_1\ge\cdots\ge \ell_k\ge 2$,
$\pi(P)=e(P)=\ell_1-1$.  We now give a simple direct proof of
$\pi(P)=e(P)$ for fans  by
showing their L-boundedness.  By comparison, the proof of
Bukh's Tree Theorem,
which is more general,
requires more elaborate probabilistic arguments.

\begin{theorem}\label{thm:palm}
Let $P$ be the fan poset $\V(\ell_1,\ldots,\ell_k)$,
with $\ell_1\ge\cdots\ge\ell_k\ge 2$.
The L-boundedness of $P$ can be classified
according to the $\ell_i$'s as follows.

\begin{description}
\item{(i)} If $k=1$ or if $\ell_1-1>\ell_2>\cdots>\ell_k$,
then $P$ is uniformly L-bounded.
\item{(ii)} If $\ell_1>\ell_2>\cdots>\ell_k$,
then $P$ is centrally L-bounded.
\item{(iii)} If $\ell_1>\ell_2$, then $P$ is  L-bounded.
\item{(iv)} If $\ell_1=\ell_2$, then $P$ is not L-bounded,
but it is lower-L-bounded.
\end{description}
\end{theorem}

\begin{proof}
Let $P$ be the fan $\V(\ell_1,\ldots,\ell_k)$ with $\ell_1\ge
\cdots\ge\ell_k\ge 2$. Note that we have $e(P)=h(P)-1= \ell_1-1$.

For (i), if $k=1$, the result follows from Erd\H{o}s's result
on chains of any given length~\cite{Erd}.
If $k\ge 1$, then $P$ is a subposet of
the harp $\Ha(\ell_1,\ell_2+1,\ldots,\ell_k+1)$, which has distinct
lengths, since $\ell_1>\ell_2+1>\cdots>\ell_k+1$. It means any
family $\F$ that is $P$-free avoids this harp.  Applying the
results on harp-free families~\cite{GriLiLu},
\[
 \hb(\F)\le e(\Ha(\ell_1,\ell_2+1,\ldots,\ell_k+1))=\ell_1-1=e(P),
\]
proving $P$ is uniformly L-bounded.

For (ii), consider any $P$-free family $\F$ of subsets of $[n]$
with $\emptyset,[n]\not\in\F$. Following~\cite{GriLiLu} apply the
{\em min partition} on the set of full chains to get blocks $\sC_A$
containing full chains $\C$ with $\min(\F\cap \C)=A$ for distinct $A$'s.
One block contains chains that avoid $\F$.  Suppose there is a block of
chains for some $A$ having $\ave_{\C\in \sC_A}|\F\cap \C|>\ell_1-1$.
Let $\F_1:=(\F\cap[A,[n]])\setminus \{A\}$, and let $Z_1$ be a chain
of size $\ell_1-1$ in $\F_1$. Such a chain exists, because
\[
\ave_{\C\in \sC_A}|\F_1\cap \C|=
\ave_{\C\in \sC_A}|\F\cap \C|-1>(\ell_1-1)-1=\ell_1-2.
\]
For successive values of $i$ from 2 to $k$, let
$\F_i:=\F_{i-1}\setminus Z_{i-1}$, and let $Z_i$ be a
chain of size $\ell_i-1$ in $\F_i$:
We show by induction on $i$ that such $Z_i$ exists.
We already have $Z_1$.  Note that for any set $S$ with
$A\subset S \subset [n]$, the proportion of full chains
in the interval $[A,[n]]$ that $S$ meets is at least
$1/(n-|A|)$.  It follows that for $i>1$,
\[
\ave_{\C\in \sC_A}|\F_i\cap \C|=
\ave_{\C\in \sC_A}|(\F_{i-1}\setminus Z_{i-1})\cap \C|
>\ell_{i-1}-2-\frac{|Z_{i-1}|}{n-|A|}>\ell_i-2.
\]
This suffices to show there exists $Z_i$ of size
$\ell_i-1$ in $\F_i$.
The chains $Z_i$ together with set $A$ form a fan $P$ in $\F$,
which contradicts our assumption that $\F$ is $P$-free.
Hence, no block $\C_A$ has
$\ave_{\C\in\sC_A}|\F\cap \C|>\ell_1-1$. It means that the
overall average $\hb(\F)\le\ell_1-1$.
Hence, $P$ is centrally L-bounded.

For (iii), it suffices to show $P$ is $m$-L-bounded for
 $m= k(\ell_1-1)$.
Let $\F$ be a $P$-free family of subsets of $[n]$ such that each set
in $\F$ has size in $[m,n-m]$. Again, apply the min partition
on $\F$.  Suppose there is a block $\sC_A$ with
$\ave_{\C\in\sC_A}|\F\cap\C|>\ell_1-1$.
It follows that $\F_1=(\F\setminus\{A\})\cap[A,[n]]$ contains
a chain $Z_1$ of size $\ell_1-1$.
Reasoning as in the proof of (ii) above it can be shown that there are
disjoint chains $Z_2,\ldots, Z_k$ in $\F_1\setminus Z_1$, each of size $\ell_2-1$,
because for $1\le j\le k$
\[
\ave_{\C\in \sC_A}|(\F_1\setminus (Z_1\cup Z_2\cdots\cup Z_j))\cap \C|>\ell_1-2-\frac{j(\ell_1-1)}{n-|A|}\ge\ell_2-2.
\]
Thus, $\F$ contains the fan $\V(\ell_1,\ell_2,\ldots,\ell_2)$,
where there are $k-1$ $\ell_2$'s, which in turn contains
$P=\V(\ell_1,\ell_2,\ldots\ell_k)$.  This contradicts that $\F$ is $P$-free.
We conclude that $\ave_{\C\in\sC_A}|\F\cap\C|\le \ell_1-1$, and
so $\hb(\F)\le \ell_1-1$. Hence $P$ is $m$-L-bounded.

Lastly, for (iv), we assume $\ell_1=\ell_2$. Since $e(P)=\ell_1-1$, the
 chains of size $\ell_1$ are  large intervals of $P$, so by
Proposition~\ref{lem:unilarge}, $P$ is not L-bounded. Suppose it
is not lower-L-bounded either. Then for some $\beta$ and
$\varepsilon$ with $\frac{1}{2}<\beta<1$ and $\varepsilon>0$, for
infinitely many $n$ there is a $P$-free family $\F$ of subsets of
$[n]$ with $\hb(\F)>e(P)+\varepsilon$, where every set in $\F$ has
size less than $\beta n$. Apply the min partition on $\F$ and let
$\sC_A$ be a block with
$\ave_{\C\in\sC_A}|\F\cap\C|>e(P)+\varepsilon$. As before, we
claim $\mathcal{F}\cap [A,[n]]$ contains $P$. By the size condition, removing a
chain of size $\ell_1-1$ from $(\F\setminus\{A\})\cap[A,[n]]$
reduces $\ave_{\C\in\sC_A}|\F\cap\C|$ by at most
$\frac{\ell_1-1}{(1-\beta)n}$. When
$n>\frac{k(\ell_1-1)}{(1-\beta)\varepsilon}$, we can find $k$
disjoint chains of size $\ell_1-1$ in
$(\F\setminus\{A\})\cap[A,[n]]$. Thus, $\F$ contains $P$, a
contradiction. Therefore, $P$ is lower-L-bounded.
\end{proof}


We mention that all fan posets $\V(\ell_1,\ldots,\ell_k)$ are
lower-L-bounded. It is also clear that a uniformly L-bounded poset is
a lower-L-bounded poset. However, there are $m$-L-bounded posets
that are not lower-L-bounded, since the size condition for
lower-L-boundedness does not exclude the small-sized subsets in the family.
We will provide such examples in next section.


Since L-boundedness and lower-L-boundedness each
imply $\pi=e$, we have completed the proof that
$\pi=e$ for fans:

\begin{corollary}
For any fan poset $P=\V(\ell_1,\ldots,\ell_k)$ with
$\ell_1\ge\cdots\ge \ell_k\ge 2$, we have $\pi(P)=e(P)=\ell_1-1$.
\end{corollary}

Although we defined $m$-L-boundedness, so far we have not
presented any poset that is $m$-L-bounded but not $(m-1)$-L-bounded,
except when $m=1$. In the following, we offer examples of posets
to show that the $m$-L-boundedness property is not vacuous.

\begin{theorem}\label{thm:Vee32}
For $m\ge 1$, the poset $P=\V(3,2,\ldots,2)$, where there are
$m+1$ $2$'s, is $m$-L-bounded but not $(m-1)$-L-bounded.
\end{theorem}

\begin{proof}
Consider the family
$\F=\binom{[n]}{n-m-1}\cup\binom{[n]}{n-m}\cup\F_0$, where $\F_0$
consists of a set in $\binom{[n]}{n-m+1}$.  For any set $F\in \F$,
at most $m+2$ sets strictly contain it, hence $\F$ is $P$-free. We
have $\hb(\F)>2=e(P)$, and so $P$ is not $(m-1)$-L-bounded.

Next consider any $P$-free family $\F$ of subsets of $[n]$ with
sizes in $[m,n-m]$. We apply the min partition on the set of full
chains. Let $\sC_A$ be any block. If $\F\cap[A,[n]]$ does not
contain a chain of size 3, then
$\ave_{\C\in \sC_A}|\F\cap \C|\le2$.
Suppose there is a chain of size
3 in $\F\cap[A,[n]]$, say $A\subset B\subset C$. There are at
most $m$ sets in $\F\cap[A,[n]]$, besides $A$, $B$, and $C$. Recall
that $\ave_{\C\in \sC_A}|\F\cap\C|$ is equal to the Lubell
function of the $P$-free family, obtained by removing the elements
of $A$ from each set in $\F\cap[A,[n]]$, in the smaller Boolean
lattice $\B_{n-|A|}$.  Here, $m+2\le n-|A|\le n-m$. By the size
condition, $\ave_{\C\in \sC_A}|\F\cap\C|$ is at most
$(1+\frac{1}{m+2}+\frac{2}{(m+1)(m+2)})+m(\frac{1}{m+2})\le 2$, as
$m\ge 1$. We conclude that $\hb(\F)\le 2$, which gives the
$m$-L-boundedness of $P$.
\end{proof}


\section{Constructing L-bounded Posets}\label{sec:CONS}

We have seen that L-bounded posets, including those that are
uniformly L-bounded or centrally L-bounded,  have  nice
properties.  This section contains methods to construct
L-bounded posets, thereby producing many more examples of
posets $P$ that satisfy the $\pi(P)=e(P)$ conjecture.
We begin with a construction using ordinal sums.

\begin{theorem}\label{thm:onetouni}
For any centrally L-bounded poset $P$, both $\op P$ and $P\po$ are
centrally L-bounded. Furthermore, $\op P\po$ is uniformly
$L$-bounded.
\end{theorem}

\begin{proof}
Let $P$ be a centrally L-bounded poset. Let $\F$ be a $(\op P)$-free
family of subsets of $[n]$  containing neither $\emptyset$ nor $[n]$.
We again apply the min partition on the set of full chains. For
any block $\sC_A$, the subfamily $(\F\setminus \{A\})\cap[A,[n]]$
is $P$-free. Hence, it contributes at most $e(P)$ to
$\ave_{\C\in\sC_A}|\F\cap \C|$. Therefore, each block has
$\ave_{\C\in\sC_A}|\F\cap \C|\le e(P)+1$. Then $\hb_n(\F)\le
e(P)+1$. On the other hand, it is clear that the union of any
$e(P)+1$ consecutive levels is $(\op P)$-free. That $P\po$ is centrally
L-bounded follows, since this property is preserved by taking the dual.
The first part is proved.

For the second part, consider the union $Q$ of any $e(P)+2$
consecutive levels in $\B_n$.  An interval $[A,B]$ in $Q$ contains no
more than $e(P)$ consecutive levels, strictly between $A$ and $B$.
Then by definition of $e(P)$, $[A,B]\setminus\{A,B\}$ is $P$-free.
Hence, $Q$ is $(\op P\po)$-free, and so $e(P)+2\le e(\op P\po)$.

On the other hand, let $\F$ be a $(\op P\po)$-free family.
Apply the min-max partition~\cite{GriLiLu} on $\sC_n$ to get
blocks $\sC_{A,B}$ containing full chains $\C$ with $\min(\F\cap \C)=A$
and $\max(\F\cap \C)=B$ for pairs $A\subseteq B$.
One block contains chains that avoid $\F$.
If $\sC_{A,B}$ is a block in the partition, then
$(\F\setminus\{A,B\})\cap[A,B]$ is $P$-free.
Now, $(\F\setminus\{A,B\})\cap [A,B]$ can be viewed as a $P$-free
family in $\B_{|B|-|A|}$.
Hence,  it contributes no more than $e(P)$ to the average,
 $\ave_{\C\in\sC_{A,B}}|\F\cap \C|$.
 Adding in the contributions of $A$ and $B$, this average is
 then at most $e(P)+2$.
Thus, $\hb_n(\F)\le e(P)+2\le e(\op P\po)$, for any $(\op P\po)$-free family.
So the poset $\op P\po$ is uniformly L-bounded.
\end{proof}

Figure~\ref{fig:OBO} illustrates how to obtain a uniformly L-bounded poset
from the butterfly poset using Theorem~\ref{thm:onetouni}.
Recall that the butterfly poset $\B$ is centrally L-bounded.

\begin{figure}[ht]
\begin{center}
\begin{picture}(150,50)
\put(10,20){\circle*{4}}
\put(20,20){$\oplus$}
\put(40,15){\circle*{4}}
\put(60,15){\circle*{4}}
\put(40,25){\circle*{4}}
\put(60,25){\circle*{4}}
\put(40,15){\line(0,1){10}}
\put(60,15){\line(0,1){10}}
\put(40,15){\line(2,1){20}}
\put(60,15){\line(-2,1){20}}
\put(70,20){$\oplus$}
\put(90,20){\circle*{4}}
\put(100,20){$=$}
\put(130,5){\circle*{4}}
\put(120,15){\circle*{4}}
\put(140,15){\circle*{4}}
\put(120,25){\circle*{4}}
\put(140,25){\circle*{4}}
\put(130,35){\circle*{4}}
\put(130,5){\line(1,1){10}}
\put(130,5){\line(-1,1){10}}
\put(130,35){\line(1,-1){10}}
\put(130,35){\line(-1,-1){10}}
\put(120,15){\line(0,1){10}}
\put(140,15){\line(0,1){10}}
\put(120,15){\line(2,1){20}}
\put(140,15){\line(-2,1){20}}
\end{picture}
\end{center}
\caption{The poset $\op \B \po$ is uniformly L-bounded.}\label{fig:OBO}
\end{figure}

We next introduce an operation that involves large intervals of
posets.  For $1\le i\le k$ let $P_i$ be a poset having a large
interval $I_i$ (which may not be unique).
Let $P_1\oplus_I P_2\oplus_I\cdots\oplus_I P_k$
denote the sum of the posets $P_i$, where for
$1\le i\le k-1$, we identify the maximal element of $I_i$ with the
minimal element of $I_{i+1}$.  This operation
depends on the choice of the large intervals.  However, if all
$P_i$'s are L-bounded, then this poset is unique, since no L-bounded
poset has more than one large interval.

\begin{theorem}\label{thm:posetchain}
For $k,\ell\ge0$ let $P_i$ ($1\le i\le k$) and $Q_j$ ($1\le j\le
\ell$) be L-bounded posets with large intervals $I_i$ and $I'_j$,
resp.  Further, assume each $P_i$ has $\hat{0}$ and
each $Q_j$ has $\hat{1}$.  If $k=0$ or $\ell=0$, it means the corresponding
collection is empty. Then the poset
\[
R:=Q_1\oplus_I\cdots \oplus_I Q_{\ell}\oplus_I P_1\oplus_I\cdots\oplus_I P_k
\]
is L-bounded.
\end{theorem}

\begin{proof}
Suppose each $P_i$ is $m_i$-L-bounded and each $Q_j$ is
$m'_j$-L-bounded. Let
$m:=\max\{m_1,\ldots,m_k,m'_1,\ldots,m'_{\ell}\}$. Hence, all
$P_i$'s and $Q_j$'s are $m$-L-bounded. We show $R$ is
$m$-L-bounded.

First suppose $\ell=0$. We use induction on the number
$k$. For $k\le 1$, this is trivial. Suppose the theorem holds for
some $k\ge1$. Let $P_1,\ldots,P_{k+1}$ be $m$-L-bounded posets
such that each $P_i$ has $\hat{0}$ and large interval $I_i$. By
induction, $P:=P_1\oplus_I\cdots\oplus_I P_{k}$ is $m$-L-bounded.
Furthermore, $e(P)=\sum_{i=1}^{k}e(P_i)$ and $e(P\oplus_I
P_{k+1})=\sum_{i=1}^{k+1}e(P_i)$, by Lemma~\ref{lem:parameters}.
Also, note that $I_1\oplus_I\cdots\oplus_I I_{k}$ is a large
interval in $P$. Consider any $(P\oplus_I P_{k+1})$-free family
$\F$ of subsets of $[n]$, where for each $F\in \F$, $m\le|\F|\le n-m$.
We partition $\F$ into $\F_1$ and $\F_2$, as we did in
Lemma~\ref{lem:parameters}, so that $\F_1$ is $P$-free and $\F_2$
is $P_{k+1}$-free.  This gives $\hb(\F_1)\le e(P)$ and
$\hb(\F_2)\le e(P_{k+1})$, and so, $\hb(\F)\le
e(P_{k+1})+e(P)=e(P\oplus_I P_{k+1})$. Therefore, $P\oplus_I
P_{k+1}$ is $m$-L-bounded.

If $k=0$, then we can use the same induction
argument to show $Q:=Q_1\oplus_I\cdots \oplus_I Q_{\ell}$ is
$m$-L-bounded by considering the dual case.

Finally, suppose $k,\ell>0$. Let $P$ and $Q$ be as above.
Consider a $(Q\oplus_I P)$-free family $\F$ of subsets of sizes in
$[m,n-m]$.  Partition $\F$ into $\F_1$ and $\F_2$, such that
$\F_1$ is $Q$-free, and $\F_2$ is $P$-free, as before. Since $P$
and $Q$ are both $m$-L-bounded, we have
\[\hb(\F)=\hb(\F_1)+\hb(\F_2)\le e(P)+e(Q)=e(P\oplus_I Q).\]
This proves that $R=P\oplus_I Q$ is $m$-L-bounded.
\end{proof}

As an example, the poset $\D_3$ is uniformly L-bounded and is
itself a large interval. Now identify the maximal element in each
$\D_3$ with the minimal element in another copy of $\D_3$,
for several consecutive  $\D_3$'s,
as shown in Figure~\ref{fig:DPD}.
This produces a ``diamond-chain'' that is uniformly L-bounded .

\begin{figure}[h]
\begin{center}
\begin{picture}(40,85)
\put(20,0){\circle*{4}}
\put(0,10){\circle*{4}}
\put(40,10){\circle*{4}}
\put(20,10){\circle*{4}}
\put(20,20){\circle*{4}}
\put(20,0){\line(2,1){20}}
\put(20,0){\line(-2,1){20}}
\put(20,20){\line(2,-1){20}}
\put(20,20){\line(-2,-1){20}}
\put(20,0){\line(0,1){20}}
\put(0,30){\circle*{4}}
\put(40,30){\circle*{4}}
\put(20,30){\circle*{4}}
\put(20,40){\circle*{4}}
\put(20,20){\line(2,1){20}}
\put(20,20){\line(-2,1){20}}
\put(20,40){\line(2,-1){20}}
\put(20,40){\line(-2,-1){20}}
\put(20,20){\line(0,1){20}}
\put(0,70){\circle*{4}}
\put(40,70){\circle*{4}}
\put(20,60){\circle*{4}}
\put(20,70){\circle*{4}}
\put(20,80){\circle*{4}}
\put(20,60){\line(2,1){20}}
\put(20,60){\line(-2,1){20}}
\put(20,80){\line(2,-1){20}}
\put(20,80){\line(-2,-1){20}}
\put(20,60){\line(0,1){20}}
\put(19,45){$\vdots$}
\end{picture}
\end{center}
\caption{The poset $\D_3\oplus_I\cdots\oplus_I \D_3$
is uniformly L-bounded.}\label{fig:DPD}
\end{figure}
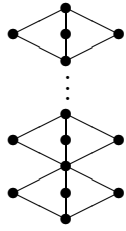


In Section~\ref{sec:FAN}, we mentioned that there are L-bounded
posets that are not lower-L-bounded. For any $m\ge 1 $, consider
$P_1=\V(3,2,\ldots,2)$ , where there are $m+1$ 2's.
Let  $P_2$ be the dual of $P_1$.
By Theorem~\ref{thm:posetchain}, the ``crab poset",
$P=P_2\oplus_I P_1$, illustrated in Figure~\ref{fig:CRB},
is $m$-L-bounded.
The family
$\F=\bigcup_{i=0}^4\binom{[n]}{i}$ is $P$-free, and
$\hb(\F)=5>e(P)$ for any $n\ge 4$. Thus, $P$ is not
lower-L-bounded.

\begin{figure}[h]
\begin{center}
\begin{picture}(100,90)
\put(50,45){\circle*{5}}
\put(30,65){\circle*{5}}
\put(20,65){\circle*{5}}
\put(10,85){\circle*{5}}
\put(40,65){$\cdots$}
\put(70,65){\circle*{5}}
\put(80,65){\circle*{5}}
\put(50,45){\line(1,1){20}}
\put(50,45){\line(3,2){30}}
\put(50,45){\line(1,2){8}}
\put(50,45){\line(0,1){16}}
\put(50,45){\line(-1,2){8}}
\put(50,45){\line(-1,1){20}}
\put(50,45){\line(-3,2){30}}
\put(20,65){\line(-1,2){10}}
\put(30,75){$\overbrace{\hspace{1.8cm}}$}
\put(45,90){$m+1$}
\put(30,25){\circle*{5}}
\put(20,25){\circle*{5}}
\put(10,5){\circle*{5}}
\put(40,20){$\cdots$}
\put(70,25){\circle*{5}}
\put(80,25){\circle*{5}}
\put(50,45){\line(1,-1){20}}
\put(50,45){\line(3,-2){30}}
\put(50,45){\line(1,-2){8}}
\put(50,45){\line(0,-1){16}}
\put(50,45){\line(-1,-2){8}}
\put(50,45){\line(-1,-1){20}}
\put(50,45){\line(-3,-2){30}}
\put(20,25){\line(-1,-2){10}}
\put(30,17){$\underbrace{\hspace{1.8cm}}$}
\put(45,0){$m+1$}
\end{picture}
\end{center}
\caption{An L-bounded poset that is not lower-L-bounded.}\label{fig:CRB}
\end{figure}
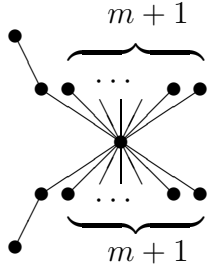


We can construct new L-bounded posets not only ``vertically'' as
above but also ``horizontally'':  Let $P_1,\ldots,P_k$ be posets
with $\hat{0}$. Then define the {\em wedge\/}
$\V(P_1,\ldots,P_k)$ to be the poset obtained by identifying the
$\hat{0}$'s of the posets.  The fan poset we introduced earlier is
the special case where the posets $P_i$ are paths.

\begin{lemma}\label{lem:eofjoin}
Let $P_1,\ldots,P_k$ be posets with $\hat{0}$, ordered so that
$e_1 \ge \cdots \ge e_k$, where $e_i=e(P_i)$.
Let $P$ be the wedge $\V(P_1,\ldots,P_k)$.
Then, $e(P)=e_1$.
\end{lemma}

\begin{proof}
We have $e(P)\ge e_1$, since $P$ contains $P_1$.
Let $n$ be large enough so that for all $i$, there exists
integer $s_i$ such that the family
$\F_i=\bigcup_{j=0}^{e_i}\binom{[n]}{s_i+j}$
contains $P_i$ as a subposet. We claim that the family
$\F=\bigcup_{j=0}^{e_1}\binom{[kn]}{s+j}$ contains $P$,
where $s=s_1+\cdots+s_k$.

For each $\F_i$ we relabel the elements in the underlying set by
$1+(i-1)n,2+(i-1)n,\ldots,in$. Let $A_i\in\F_i$ be the $\hat{0}\in
P_i$. For any $i$, if $S\in \F_i$ is an element $p\in P_i$, then
the set $(S\cup A_1\cup\cdots \cup A_k)\in \F$  will be the
element $p\in P_i\subset P$. This shows that $\F$ contains $P$.
Hence, $e(P)\le e_1$.
\end{proof}


\begin{theorem}\label{thm:joinULB}
Let $k\ge2$.  Let $P_1,\ldots,P_k$ be uniformly L-bounded posets
with $\hat{0}$, such that $e_1\ge \cdots \ge e_k$, where
$e_i=e(P_i)$. Then the wedge $P=\V(P_1,\ldots,P_k)$ is
lower-L-bounded. In addition, if $e_1>e_2$, then $P$ is L-bounded.
\end{theorem}

\begin{proof} This is a generalization of Theorem~\ref{thm:palm} (iii)
and (iv), and the proof is similar.
Here we present the case $e_1>e_2$,
but omit the details of the proof of lower-L-boundedness.
We have $e(P)=e_1$ by Lemma~\ref{lem:eofjoin}. Let
$p_i=|P_i|$, and put
$m=\sum_{i=1}^{k-1}(p_i-1)$.

Consider any $P$-free family $\F$ of
subsets of $[n]$ with sizes in $[m,n-m]$.
Suppose $\hb(\F)\le e_1$ does not hold. Apply the min partition on $\sC_n$, and
let $\sC_A$ be a block with $\ave_{\C\in\sC_A}|\F\cap\C|>e_1$.
Because $P_1$ is uniformly L-bounded, we conclude that $\F\cap
[A,[n]]$ contains $P_1$ as as subposet.
Define $\F_1=(\F\setminus\{A\})\cap[A,[n]]$.
Now let $\G_1$ be a subfamily of $\F_1$ with size $p_1-1$,
such that $\G_1\cup\{A\}$ contains $P_1$ as a subposet,
and such that $A$ is the $\hat{0}$.
It is straightforward to find disjoint subfamilies $\G_i\subseteq \F_i$,
where $\F_{i+1}=\F_i\setminus\G_i$, such that each $\G_i$
together with $A$ contains $P_i$ as a subposet.
This is because, for $2\le i\le k$,
\[
\ave_{\C\in \sC_A}|(\F_i\cup \{A\})\cap \C|
>e_1-\sum_{j=1}^{i-1}(p_j-1)/(n-|A|)>e_2\ge e_i,
\]
 and because $P_i$ is uniformly L-bounded. Then
$(\bigcup_{i=1}^k\G_i)\cup\{A\}$ contains $P=\V(P_1,\ldots,P_k)$,
which is impossible. Thus,  we must have
$\ave_{\C\in\sC_A}|\F\cap\C|\le e_1$ for every block
$\sC_A$, and so $\hb(\F)\le e_1$.
It means that $P$ is $m$-L-bounded.
\end{proof}



\Remark Burcsi and Nagy~\cite{BurNag} define a similar
construction method to produce a class of posets that
satisfies $\pi(P)=e(P)=\frac{|P|+h(P)}{2}$.
This class includes some of our L-bounded posets.

\section{Constructions with lower-L-bounded posets}\label{sec:BLUP}

The operations above on L-bounded posets  are
also useful for constructing lower L-bounded posets,
which then gives us additional new posets satisfying
the $\pi=e$ conjecture.  Note that by duality
we can obtain similar results for upper-L-bounded posets.  The
first two results are analogous to those in the last section.

\begin{theorem}\label{thm:LLBchain}
For $1\le i\le k$, let $P_i$ be a lower-L-bounded poset with
$\hat{0}$ and with a large interval $I_i$. Then
\[
P_1\oplus_I\cdots\oplus_I P_k
\]
is lower-L-bounded.
\end{theorem}

\begin{proof}
By iterating the argument it suffices to prove this for $k=2$.
Let $P=P_1\oplus_I P_2$. By
Lemma~\ref{lem:parameters}, we have $e(P)=e(P_1)+e(P_2)$.
Let $\beta\in(\frac{1}{2},1)$ and $\varepsilon>0$.
For $i=1,2$, since $P_i$ is lower-L-bounded,
there exists
$N_i=N_i(\beta,\frac{\varepsilon}{2})$
such that for all $n\ge N_i$,
every $P_i$-free family $\F$, containing subsets of size at most $\beta n$,
satisfies $\hb(\F)< e(P_i)+\frac{\varepsilon}{2}$.
Set $N=\max\{N_1,N_2\}$, and suppose $n\ge N$.
Consider a $P$-free family $\F$ of subsets of $[n]$, each of size at most $\beta n$.
As in the proof of Lemma~\ref{lem:parameters}, one can
split $\F$ into a $P_1$-free $\F_1$ and a $P_2$-free $\F_2$.
We have $\hb(\F_i)<e(P_i)+\frac{\varepsilon}{2}$ for $i=1,2$.
Then $\hb(\F)=\hb(\F_1)+\hb(\F_2)<e(P_1)+e(P_2)+\varepsilon=e(P)+\varepsilon$.
This means that $P$ is lower L-bounded.
\end{proof}

\begin{theorem}\label{thm:joinLLB}
For $1\le i\le k$, let $P_i$ be a lower-L-bounded poset with
$\hat{0}$. Then the wedge $P=\V(P_1,\ldots,P_k)$ is
lower-L-bounded.
\end{theorem}

\begin{proof}
Again we only need to show the case $k=2$. Let $e_i=e(P_i)$.
We may assume $e_1\ge e_2$. By Lemma~\ref{lem:eofjoin}, $e(P)=e_1$.
Given $\beta\in(\frac{1}{2},1)$ and $\varepsilon>0$, for $i=1,2$, there exists
$N_i$ such that if $n\ge N_i$ then any $P_i$-free family $\F$ containing
subsets of sizes at most $\beta n$ will satisfy $\hb(\F)< e_i+\frac{\varepsilon}{2}$.
Let $N=\max\{\frac{N_1}{1-\beta},\frac{N_2}{1-\beta},\frac{2(|P_i|-1)}{\varepsilon(1-\beta)} \}$.
For $n\ge N$, consider a $P$-free family $\F$ of subsets of $[n]$ containing subsets of
sizes at most $\beta n$. Apply the min partition on $\sC_n$.
We claim $\ave_{\C\in\sC_A}|\F\cap\C|<e_1+\varepsilon$.

If $\F\cap [A,[n]]$ is $P_1$-free, then
$\ave_{\C\in\sC_A}|\F\cap\C|=\hb_{n'}(\F')$,
where $n'=n-|A|$ and $\emptyset\in\F':=\{F\setminus A\mid
F \in\F,A\subseteq F\}$. The new family $\F'$ is $P$-free.
Note that $n'\ge(1-\beta)n\ge N_1$, and $\F'$ contains subsets
of sizes at most $\beta n-|A|<\beta(n-|A|)=\beta n'$.
So, $\hb_{n'}(\F')<e_1+\frac{\varepsilon}{2}$.

Else, suppose $\F\cap [A,[n]]$ contains $P_1$.
One can find a subfamily $\G\cup \{A\}\subseteq \F\cap [A,[n]]$
containing $P_1$ such that $|\G|=|P_1|-1$.
Then $(\F\setminus \G)\cap[A,[n]]$ must be $P_2$-free, since
$\F\cap[A,[n]]$ cannot contain $P$.
By the same reasoning as above, we have
$\ave_{\C\in\sC_A}|(\F\setminus \G)\cap\C|<e_2+\frac{\varepsilon}{2}$.
The contribution of $\G$ to
$\ave_{\C\in\sC_A}|\F\cap\C|$ is at most
$\frac{|P_1|-1}{n'}<\frac{\varepsilon}{2}$.
Hence,
$\ave_{\C\in\sC_A}|\F\cap\C|<e_2+\frac{\varepsilon}{2}+\frac{\varepsilon}{2}
\le e_1+\varepsilon$.

Therefore, $\hb(F)<e_1+\varepsilon$, and $P$ is lower-L-bounded.
\end{proof}

\noindent{\bf Examples.} We have seen that the fan posets
$\V(2,2)$ and $\V(2,2,2)$ are lower-L-bounded.
The left poset in
Figure~\ref{fig:trees} is obtained by identifying the
$\hat{0}$ of $\V(2,2)$ with the maximal element of a large interval
of another $\V(2,2)$. The middle poset
is similar, but each $\V(2,2)$ is replaced by $\V(2,2,2)$.
The right poset is obtained by wedging the
other two posets. Each poset is lower-L-bounded by the
theorems above.

\begin{figure}[h]
\begin{center}
\begin{picture}(300,60)
\put(50,0){\circle*{4}}
\put(50,0){\line(1,1){20}}
\put(50,0){\line(-1,1){20}}
\put(30,20){\circle*{4}}
\put(70,20){\circle*{4}}
\put(70,20){\line(1,1){20}}
\put(70,20){\line(-1,1){20}}
\put(90,40){\circle*{4}}
\put(50,40){\circle*{4}}
\put(150,0){\circle*{4}}
\put(150,0){\line(0,1){20}}
\put(150,0){\line(-5,4){25}}
\put(150,0){\line(5,4){25}}
\put(125,20){\circle*{4}}
\put(175,20){\circle*{4}}
\put(150,20){\circle*{4}}
\put(150,20){\line(0,1){20}}
\put(150,20){\line(-5,4){25}}
\put(150,20){\line(5,4){25}}
\put(150,40){\circle*{4}}
\put(125,40){\circle*{4}}
\put(175,40){\circle*{4}}
\put(240,20){\circle*{4}}
\put(240,40){\circle*{4}}
\put(220,20){\circle*{4}}
\put(220,40){\circle*{4}}
\put(250,0){\line(1,2){10}}
\put(250,0){\line(-1,2){10}}
\put(250,0){\line(-3,2){30}}
\put(240,20){\line(0,1){20}}
\put(240,20){\line(-1,1){20}}
\put(250,0){\circle*{4}}
\put(250,0){\line(1,2){10}}
\put(250,0){\line(1,1){20}}
\put(250,0){\line(3,2){30}}
\put(260,20){\circle*{4}}
\put(270,20){\circle*{4}}
\put(280,20){\circle*{4}}
\put(260,20){\line(0,1){20}}
\put(260,20){\line(1,2){10}}
\put(260,20){\line(1,1){20}}
\put(260,40){\circle*{4}}
\put(270,40){\circle*{4}}
\put(280,40){\circle*{4}}
\end{picture}
\end{center}
\caption{Three lower-L-bounded trees.}\label{fig:trees}
\end{figure}
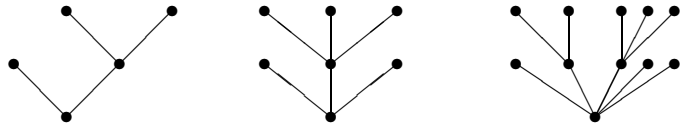

Each tree in the figure is obtained by repeatedly applying
Theorems~\ref{thm:LLBchain} and~\ref{thm:joinLLB} to poset
$\Pa_2$. We can build more elaborate examples by expanding
 each $\Pa_2$ in the tree posets by L-bounded posets. For
instance, the poset in Figure~\ref{fig:BLUP} is
the wedge $\V(P_1,P_2,P_3)$,
where $P_1=\Pa_2$ , $P_2=\D_3\oplus_I \V(\Ha(4,3),\Pa_2)$, and
$P_3=\Pa_2\oplus_I\V(2,2,2)$. This poset is lower-L-bounded.

\begin{figure}[h]
\begin{center}
\begin{picture}(110,60)
\put(50,0){\circle*{4}}
\put(10,20){\circle*{4}}
\put(40,10){\circle*{4}}
\put(50,10){\circle*{4}}
\put(50,20){\circle*{4}}
\put(60,10){\circle*{4}}
\put(90,20){\circle*{4}}
\put(50,30){\circle*{4}}
\put(50,40){\circle*{4}}
\put(30,40){\circle*{4}}
\put(40,50){\circle*{4}}
\put(60,40){\circle*{4}}
\put(90,40){\circle*{4}}
\put(80,40){\circle*{4}}
\put(100,40){\circle*{4}}
\put(50,0){\line(0,1){20}}
\put(50,0){\line(2,1){40}}
\put(50,0){\line(1,1){10}}
\put(50,0){\line(-1,1){10}}
\put(50,0){\line(-2,1){40}}
\put(50,20){\line(-1,-1){10}}
\put(50,20){\line(1,-1){10}}
\put(50,20){\line(1,2){10}}
\put(50,20){\line(0,1){20}}
\put(50,20){\line(-1,1){20}}
\put(90,20){\line(0,1){20}}
\put(90,20){\line(1,2){10}}
\put(90,20){\line(-1,2){10}}
\put(40,50){\line(-1,-1){10}}
\put(40,50){\line(1,-1){10}}
\end{picture}
\end{center}
\caption{The poset $\V(\Pa_2,\D_3\oplus_I \V(\Ha(4,3),\Pa_2),
\Pa_2\oplus_I\V(2,2,2)$).}\label{fig:BLUP}
\end{figure}

The posets in the figures have a tree-structure rooted at the
bottom.  We may combine such a poset with its dual, joining them
at the roots.   The {\em baton} poset~\cite{GriLu}
$\Pa_k(s,t)=d(\V_s)\oplus_I \Pa_k\oplus_I \V_t$ is special case of
this result.

\begin{theorem}\label{thm:LLBdLLB}
Let $P_1$ and $P_2$ be lower-L-bounded posets with $\hat{0}$.  Let
$d(P_2)$ be the dual of $P_2$. Then poset obtained by identifying
the $\hat{0}$ of $P_1$ to the $\hat{1}$ of $d(P_2)$ satisfies
$\pi=e$.
\end{theorem}

\begin{proof}
Since a family $\F$ of subsets of $[n]$ is $P_2$-free if and only if
the family $\F'=\{F\mid([n]\setminus F)\in\F\}$ is $d(P_2)$-free,
 and since $P_2$ is lower-L-bounded, we have
\[
\pi(d(P_2))=\lim_{n\rightarrow\infty}\frac{\La(n,d(P_2))}{\nchn}=
\lim_{n\rightarrow\infty}\frac{\La(n,P_2)}{\nchn}=e(P_2)=e(d(P_2)).
\]
By Lemma~\ref{lem:parameters}, we have
\[
e(P_1)+e(d(P_2))=e(P)\le \pi(P)
\le \pi(P_1)+\pi(d(P_2))=e(P_1)+e(d(P_2)).
\]
Hence $\pi(P)=e(P)$.
\end{proof}


\section{Concluding Remarks}\label{sec:REM}

While our methods verify Conjecture~\ref{conj:GriLu} for many new
posets, it is still far from proven.   Beyond the conjecture, there are
 problems on Lubell bounded posets that are  interesting in
their own right.

\begin{question}
Is it true for every poset $P$ that the
Lubell function limit  $\lambda(P)$ exists?
\end{question}

We believe it does exist.

Recall that uniformly L-bounded posets $P$ have the property
\[
e(P)=\pi(P)=\lambda(P).
\]
We suspect that
only uniformly L-bounded posets satisfy this.
For the (larger) class of centrally L-bounded posets $P$,
even if $\lambda(P)$ exists, it may be larger than $\pi(P)$.
A good example is the butterfly
poset $\B$. We have seen a $\B$-free family $\F$ with $\hb(\F)=3$.
On the other hand, the butterfly $\B$ is a subposet
of $\Pa_4$. Thus, $|\F\cap\C|\le 3$ for any $\B$-free family $\F$.
We conclude that $\lambda_n(\B)=3$ for all $n\ge 2$, and hence,
$\lambda(\B)=3$, while $\pi(\B)=2$.

Another centrally L-bounded example is the fan $\J=\V(3,2)$ and the family
\[
\F=\{[n]\}\cup\{[n]\setminus\{i\}\mid i\mbox{ is odd.}\}
\cup\{[n]\setminus\{i,j\}\mid \mbox {At least one of }i,j\mbox{ is
odd.}\}
\]
Then $\F$ is $\V(3,2)$-free, since every set in $\F$ has at most
two supersets.
For all $n$ we have $\lambda_n(\V(3,2))\ge\hb(\F)\ge
\frac{9}{4}>\pi(\V(3,2))=2$.

Posets $P$ with more than one large interval also fail to satisfy
$e(P)=\pi(P)=\lambda(P)$:
If $P$ contains two large intervals, they cannot share
both their maximal and minimal elements, and then either
$\bigcup_{i=0}^{e}\binom{[n]}{i}$ or
$\bigcup_{i=0}^{e}\binom{[n]}{n-i}$ is $P$-free.
It means that $\lambda(P)>e(P)$, if $\lambda(P)$ exists.

Our general question is this:

\begin{question}
Do there exist posets $P$ that are not uniformly L-bounded, but
satisfy $\lambda(P)=e(P)$?
\end{question}

The idea of $m$-L-boundedness is to bound the Lubell function for
$P$-free families $\F$ of subsets, where the sizes of the sets in $\F$
are restricted to be in the range $[m,n-m]$, that is, we toss away
the comparatively few subsets that are very small or very large.
In fact, we could tighten this size restriction even further, while
still giving up on comparatively few subsets.

We generalize the Lubell measure $\lambda_n(P)$ in the following way:
For a function $f:\mathbb{N}\rightarrow\mathbb{N}$,
we consider all $P$-free families $\F$ of sets, all with sizes
in the range $[f(n),n-f(n)]$, and define $\lambda_{n}^{(f(n))}(P)$
to be the maximum value of $\hb_n(\F)$ over all these families.
To prove the Griggs-Lu Conjecture~\ref{conj:GriLu},
it suffices to find for each poset $P$ a function
$f$ that satisfies both
\begin{center}
(1) $\limsup_{n\rightarrow\infty}\lambda_{n}^{(f(n))}(P)\le e(P)$
and (2) $\sum_{i=0}^{f(n)}\binom{n}{i}=o(2^n)$.
\end{center}
In this paper, our goal is achieved for L-bounded posets $P$, for which
the constant function $f(n)\equiv m$ works, for constant $m$ depending on $P$.

Let $f(n)=\lfloor \alpha n\rfloor$ for some constant $\alpha\in(0,\frac{1}{2})$.
Then $f$ satisfies condition (2) by Shannon's Theorem~\cite[page 256]{AloSpe}.
For lower-L-bounded and upper-L-bounded posets $P$,
$f(n)=\lfloor \alpha n\rfloor$ also satisfies
condition (1), hence $\pi(P)=e(P)$.
Incidentally,
it would be interesting if one could find examples of posets $P$ with
$\limsup_{n\rightarrow\infty}\lambda_{n}^{(\lfloor\alpha n\rfloor)}(P)> e(P)$.
Note that this would not disprove Conjecture~\ref{conj:GriLu}.

Remarkably, there are posets $P$, such as $\V_2$, with
$\limsup_{n\rightarrow\infty}\lambda_{n}^{(f(n))}(P)> e(P)$,
for any constant function $f$.
In Section~\ref{sec:LBP},  we constructed $\V_2$-free families $\F_n$ showing
$\lambda_n^{(m)}(\V_2)\ge\hb_n(\overline{\F}_n)>1$.
In fact,
$\limsup_{n\rightarrow\infty}\lambda_n^{(m)}(\V_2)\ge
\limsup_{n\rightarrow\infty}\hb_n(\overline{\F}_n)= 1+(\frac{1}{2})^m$.
On the other hand, one can use the min partition method to show that $\lambda_n^{(m)}(\V_2)\le 1+\frac{1}{m+1}$.
So $\lim_{m\rightarrow\infty}(\limsup_{n\rightarrow\infty}\lambda_{n}^{(m)}(\V_2))=1$, which is $e(\V_2)$.

More generally, for any poset $P$, $\lambda_{n}^{(m)}(P)\ge\lambda_{n}^{(m+1)}(P)$ holds for all $m$,
since in the latter term we consider fewer $P$-free families.
Therefore,
$\lim_{m\rightarrow\infty}(\limsup_{n\rightarrow\infty}\lambda_{n}^{(m)}(P))$ exists.
We are interested in this limit.

\begin{question}
Does there exist a poset $P$ such that
\[\lim_{m\rightarrow\infty}(\limsup_{n\rightarrow\infty}\lambda_{n}^{(m)}(P))>e(P)?\]
\end{question}

For L-bounded posets, our strong suspicion is that for large enough $n$,
the largest $P$-free families of subsets of $[n]$ cluster near the
middle ranks, where most subsets are located.  Specifically, we ask

\begin{question}
For any $m$-L-bounded poset $P$, does there exist $N=N(m,e,P)$
such that for all $n\ge N$, $\lanp=\Sigma(n,e(P))$?
If true, does there exist such $N$, so that for all $n\ge N$,
a largest $P$-free family of subsets of $[n]$ must be $\B(n,e(P))$?
\end{question}

For $m=0$, both parts hold for all $n$.
For $m=1$, we proved in Theorem~\ref{thm:CLB}
that if $P$ is centrally L-bounded,
the first answer is yes with $N=e(P)+3$.
The second answer is open for $m=1$, though it is yes for the
particular example of the butterfly poset.
For general $m$, when $P$ is
$m$-L-bounded, we cannot yet establish this value of $\lanp$.
Proposition~\ref{thm:mLB}  gives the weaker general bound $\lanp\le
\Sigma(n,e(P))+2\sum_{i=0}^{m-1}\binom{n}{i}$.


\begin{thebibliography}{99}

\bibitem{AloSpe} N. Alon and J. H. Spencer,  {\em The Probabilistic Method},
3rd edition Wiley, New York (2008)


\bibitem{Buk}
B. Bukh, {Set families with a forbidden poset}, {\em Elect. J.
Combin.} {\bf 16} (2009), R142, 11p.


\bibitem{BurNag}
P. Burcsi and D. T. Nagy, {The method of double chains for largest
families with excluded subposets},  {\em Elect. J. Graph
Theory and Applications}, Vol {\bf 1} (2013), 40--49.

\bibitem{DebKat} A. De Bonis and G. O. H. Katona, {Largest families
without an $r$-fork}, {\it Order} {\bf 24} (2007), 181--191.

\bibitem{DebKatSwa}
A. De Bonis, G. O. H. Katona, and K. J. Swanepoel, {Largest family
without $A\cup B \subset C\cap D$}, {\em J. Combinatorial Theory
(Ser. A)} {\bf 111} (2005), 331--336.


\bibitem{Erd} P. Erd\H{o}s,
On a lemma of Littlewood and Offord, {\em Bull. Amer. Math. Soc.}
{\bf 51} (1945), 898--902.


\bibitem{GriKat} J. R. Griggs and G. O. H. Katona, {No
four subsets forming an $N$}, {\em J. Combinatorial Theory (Ser.
A)} {\bf 115} (2008), 677--685.

\bibitem{GriLi} J. R. Griggs and W.-T. Li, The Partition Method
for Poset-free Families, {\em J. Combinatorial Optimization},
Volume {\bf 25}, Issue {\bf 4}, (2013), 587-596.

\bibitem{GriLiLu} J. R. Griggs, W.-T. Li, and L. Lu, {Diamond-free
Families}, {\em J. Combinatorial Theory (Ser. A)} {\bf 119}
(2012), 310--322.

\bibitem{GriLu} J. R. Griggs and L. Lu, {On families of subsets
with a forbidden subposet}, {\em Combinatorics, Probability, and
Computing} {\bf 18} (2009), 731--748.

\bibitem{KatTar}  G. O. H. Katona and T. G. Tarj\'an, {Extremal problems
with excluded subgraphs in the $n$-cube},  in: M. Borowiecki, J.
W. Kennedy, and M. M. Sys\l o (eds.) {\bf Graph Theory}, \L
ag\'ow, 1981, {\em Lecture Notes in Math.}, {\bf 1018} 84--93,
Springer, Berlin Heidelberg New York Tokyo, 1983.

\bibitem{KraMarYou} L. Kramer, R. R. Martin, and M. Young, {On diamond-free
subposets of the Boolean lattice}, {\em J. Combinatorial Theory (Ser. A)} {\bf 120}
(2013), 545--560.

\bibitem{Li}W.-T. Li, Extremal problems on families of subsets with
forbidden subposets, Ph.D. dissertation, University of South
Carolina, 2011.

\bibitem{Lub}  D. Lubell, {A short proof of Sperner's lemma},
{\em J. Combin. Theory} {\bf 1} (1966), 299.


\bibitem{Spe} E. Sperner,
{Ein Satz \"{u}ber Untermengen einer endlichen Menge},
{\em Math. Z.} {\bf 27} (1928), 544--548.

\bibitem{Sta} R. P. Stanley, {\em Enumerative Combinatorics Vol. 1},
Cambridge University Press, 1997.


\end{thebibliography}
\end{document}